\documentclass{article}

\usepackage{amsmath}
\usepackage{amsthm}
\usepackage{amsfonts}
\usepackage{amssymb}
\usepackage{graphicx}
\usepackage{subfigure}
\usepackage{enumerate}

\DeclareMathOperator*{\inte}{int}
\DeclareMathOperator*{\Ang}{Ang}
\DeclareMathOperator*{\ang}{ang}

\newtheorem{thm}{Theorem}[section]
\newtheorem{cor}[thm]{Corollary}
\newtheorem{lem}[thm]{Lemma}

\theoremstyle{definition}

\theoremstyle{remark}

\begin{document}

\title{Twofold Translative Tilings with Convex Bodies}

\author{Kirati Sriamorn}

\maketitle

\begin{abstract}
  Let $K$ be a convex body. It is known that, in general, if $K$ is a $k$-fold translative tile (for some positive integer $k$), then $K$ may not be a (onefold) translative tile. However, in this paper I will show that for every convex body $K$, $K$ is a twofold translative tile if and only if $K$ is a translative tile.
\end{abstract}

\bigskip


\medskip


\section{Introduction}
Let $D$ be a connected subset of $\mathbb{R}^n$, and let $k$ be a positive integer. We say that a family of convex bodies $\{K_1,K_2,\ldots\}$ is a \emph{$k$-fold tiling of $D$}, if each point of $D$ which dose not lie in the boundary of any $K_i$, belongs to exactly $k$ convex bodies of the family.

Let $K$ be an $n$-dimensional convex body, and let $X$ be a discrete multisubset of $\mathbb{R}^n$. Denote by $K+X$ the family
\begin{equation*}
\{K+x:~x\in X\},
\end{equation*}
where $K+x=\{y+x:~y\in K\}$.
We say that the family $K+X$ is a \emph{$k$-fold translative tiling with $K$}, if $K+X$ is a $k$-fold tiling of $\mathbb{R}^n$. In particular, if $X=\Lambda$ is a lattice, then $K+\Lambda$ is called a \emph{$k$-fold lattice tiling with $K$}. We call $K$ a \emph{$k$-fold translative (lattice) tile} if there exists a $k$-fold translative (lattice) tiling with $K$. A onefold tiling (tile) is simply called a tiling (tile).

Let $P$ be an $n$-dimensional centrally symmetric polytope with centrally symmetric facets. A \emph{belt} of $P$ is the collection of its facets which contain a translate of a given subfacet ($(n-2)$-face) of $P$.

For the case of onefold tilings, Venkov \cite{venkov} and McMullen \cite{mcmullen} proved the following result.
 \begin{thm}\label{mcmullen thm}
 A convex body $K$ is a translative tile if and only if $K$ is a centrally symmetric polytope with centrally symmetric facets, such that each belt of $K$ contains four or six facets.
 \end{thm}
 Furthermore, a consequence of the proof of this result is that, every convex translative tile is also a lattice tile.
 In the case of general $k$-fold tilings, Gravin, Robins and Shiryaev \cite{gravin} proved that
\begin{thm}\label{gravin thm}
If a convex body $K$ is a $k$-fold translative tile, for some positive integer $k$, then $K$ is a centrally symmetric polytope with centrally symmetric facets.
 \end{thm}
 Moreover, they also showed that, every rational polytope $P$ that is centrally symmetric and has centrally symmetric facets must be a $k$-fold lattice tile, for some positive integer $k$. This result implies that there exists a polytope $P$ such that $P$ is a $k$-fold translative tile (for some $k>1$), but $P$ is not a translative tile. For example, the octagon shown in Fig. \ref{octagon} is a $7$-fold lattice tile, but is not a translative tile.

\begin{figure}[!ht]
  \centering
    \includegraphics[scale=1]{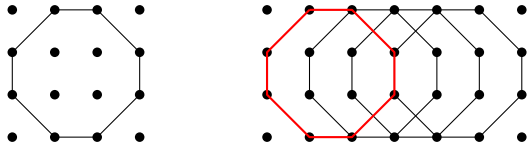}
   \caption{The octagon that is a $7$-fold lattice tile}\label{octagon}
\end{figure}

In this paper, I will prove the following surprising result:
\begin{thm}\label{main thm}
A convex body $K$ is a twofold translative tile if and only if $K$ is a onefold translative tile.
\end{thm}

In order to prove this result, I will modify the method used in \cite{mcmullen}. As an immediate consequence of Theorem \ref{main thm}, we have
\begin{cor}
A convex body $K$ is a twofold translative tile if and only if $K$ is a twofold lattice tile.
\end{cor}

\section{Some Notations}
Let $K$ be an $n$-dimensional convex body, and let $X$ be a discrete multisubset of $\mathbb{R}^n$ which contains the origin. Let $q$ be a point on the boundary $\partial K$ of $K$. We define
\begin{equation*}
X_K(q)=\{x\in X:~q\in K+x\},
\end{equation*}
\begin{equation*}
\overset{\circ}{X}_K(q)=\{x\in X_K(q):~q\in\inte(K+x)\},
\end{equation*}
and
\begin{equation*}
\partial X_K(q)=\{x\in X_K(q):~q\in\partial(K+x)\}.
\end{equation*}
In addition, we define
\begin{equation*}
\overline{\partial} X_K(q)=\{x\in \partial X_K(q):~\inte(K)\cap(K+x)=\emptyset\},
\end{equation*}
(see Fig. \ref{XK} for an example).

\begin{figure}[!ht]
  \centering
    \includegraphics[scale=.9]{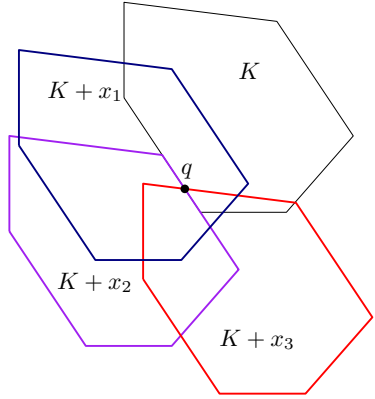}
   \caption{$X_K(q)=\{0,x_1,x_2,x_3\}$
   ,~$\overset{\circ}{X}_K(q)=\{x_1\}$,~$\partial X_K(q)=\{0,x_2,x_3\}$ and $\overline{\partial} X_K(q)=\{x_2\}$}
   \label{XK}
\end{figure}

Now suppose that $P$ is an $n$-dimensional centrally symmetric convex polytope with centrally symmetric facets. Let $G$ be a translate of a subfacet of $P$. Denote by $\mathcal{B}_P(G)$ the belt of $P$ determined by $G$. Let $q$ be a point that lies in a facet in $\mathcal{B}_P(G)$. Let $S(G,q)$ be the $(n-2)$-dimensional plane that contains the point $q$ and parallels to $G$. We define
\begin{equation*}
\dot{\partial} X_P(G,q)=\{x\in \overline{\partial} X_P(q):~P\cap(P+x)\subset S(G,q)\}.
\end{equation*}
Let $F$ be a subset of $\partial P$ containing the point $q$, we define
\begin{equation*}
\overline{\partial}X_P(G,F,q)=\{x\in\overline{\partial} X_P(q):~S(G,q)\cap F\cap(P+x)\neq F\cap(P+x)\},
\end{equation*}
(see Fig. \ref{XP}).
\begin{figure}[!ht]
  \centering
    \includegraphics[scale=.9]{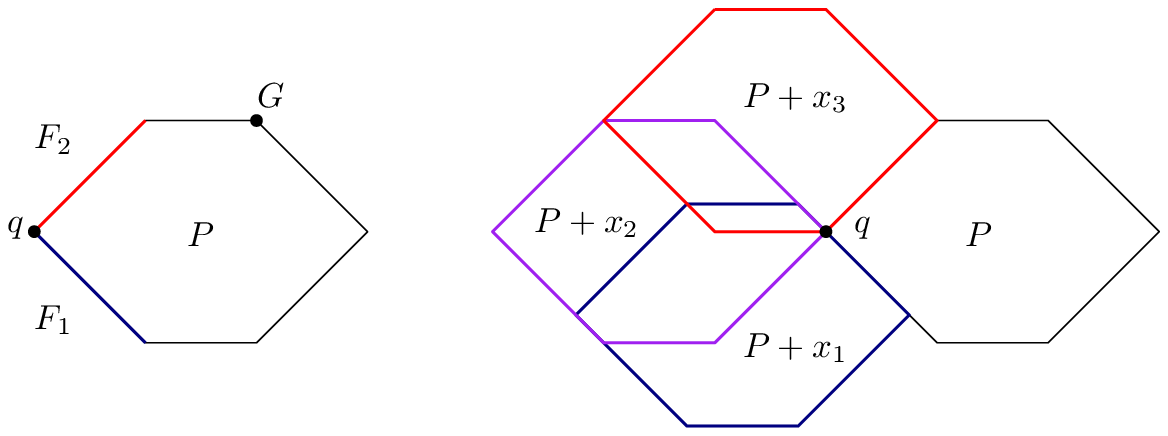}
   \caption{$\dot{\partial} X_P(G,q)=\{x_2\}$,~ $\overline{\partial}X_P(G,F_1,q)=\{x_1\}$
    and $\overline{\partial}X_P(G,F_2,q)=\{x_3\}$}
   \label{XP}
\end{figure}

Let $F_P(G,q)$ be the union of those facets in $\mathcal{B}_P(G)$ which contain $q$. It is easy to see that, $S(G,q)$ divides $F_P(G,q)$ into two parts. After choosing a direction, we may define these two parts $F_P^+(G,q)$ and $F_P^-(G,p)$ as shown in Fig. \ref{FP}.
Denote by $\Ang_P(G,q)$ the angle from $F_P^+(G,q)$ to $F_P^-(G,q)$, and denote by $\ang_P(G,q)$ the measure of $\Ang_P(G,q)$ in radian. Obviously, if $q$ lies in some subfacet that parallels to $G$, then $\ang_P(G,q)<\pi$, otherwise, $\ang_P(G,q)=\pi$. We will denote by $E_P^+(G,q)$ the subfacet which parallels to $G$ and is contained in $F_P^+(G,q)$, but is not containing $q$ (Fig. \ref{FP}). We can also define $E_P^-(G,q)$ in the similar way.

\begin{figure}[!ht]
  \centering
    \includegraphics[scale=1]{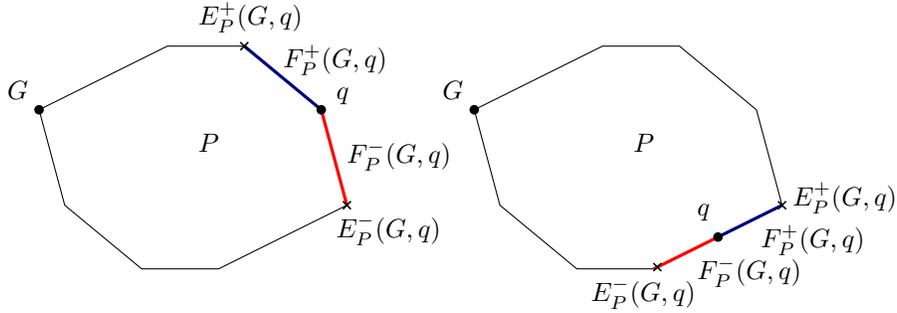}
   \caption{$F_P^+(G,q)$,~$F_P^-(G,q)$,~$E_P^+(G,q)$ and $E_P^-(G,q)$}\label{FP}
\end{figure}

\section{Some Lemmas}
For a positive real number $\varepsilon$ and a point $p$, denote by $B_\varepsilon(p)$ the closed ball with center $p$ and radius $\varepsilon$.

\begin{lem}\label{local to global}
Suppose that $K$ and $K'$ are convex bodies. If there exist a point $p\in \partial K\cap\partial K'$ and a positive real number $\varepsilon$ such that $B_\varepsilon(p)\cap\inte(K)\cap\inte(K')=\emptyset$, then $\inte(K)\cap\inte(K')=\emptyset$.
\end{lem}
\begin{proof}
Since $B_\varepsilon(p)\cap K$ and $B_\varepsilon(p)\cap K'$ are convex, by applying the basic result of Convex and Discrete Geometry, we know that there is a hyperplane $H$ which separates $B_\varepsilon(p)\cap K$ and $B_\varepsilon(p)\cap K'$. Assume that $p'\in\inte(K)\cap\inte(K')$. By the convexity, the line segment $L$ between the point $p$ and the point $p'$ must lie in $K\cap K'$. Therefore, $L\cap B_\varepsilon(p)$ must be contained in the hyperplane $H$, and hence $p'\in H$. On the other hand, there is a positive real number $\delta$ such that $B_\delta(p')\subset \inte(K)\cap\inte(K')$. So $B_\delta(p')\subset H$, this is impossible.
\end{proof}

\begin{lem}
Let $D$ be a connected subset of $\mathbb{R}^n$, and let $k$ be a positive integer. Suppose that a family of convex bodies $\{K_1,K_2,\ldots\}$ is a $k$-fold tiling of $D$. We have that, for every $i\in\{1,2,\ldots\}$ and every point $q\in \partial K_i$, if $q$ is an interior point of $D$, then there must be a $j\in\{1,2,\ldots\}$ such that $q\in \partial K_j$ and $\inte(K_i)\cap \inte(K_j)=\emptyset$.
\end{lem}
\begin{proof}
 For $i\neq j$, let
\begin{equation*}
A_i^j=\{p\in\partial K_i:~ B_\varepsilon(p)\cap \partial K_j\cap\inte(K_i)=\emptyset,~\text{for some}~\varepsilon>0\},
\end{equation*}
and
\begin{equation*}
B_i^j=\{p\in\partial K_i\setminus A_i^j:~ B_\varepsilon(p)\cap \partial K_j\subset K_i,~\text{for some}~\varepsilon>0\}.
\end{equation*}
We note that, if $p\in \partial K_i\setminus(A_i^j\cup B_i^j)$, then $p$ must lie in $\partial K_j$ and for all $\varepsilon>0$, we have $B_\varepsilon(p)\cap\partial K_j \cap \inte(K_i)\neq\emptyset$ and $(B_\varepsilon(p)\cap\partial K_j)\setminus K_i\neq\emptyset$ (Fig. \ref{ABij}). Obviously, when $K_i\cap K_j=\emptyset$, we have $A_i^j=\partial K_i$ and $B_i^j=\emptyset$. We note that for a fixed $i$, there are finitely many $j$ such that $K_i\cap K_j\neq \emptyset$. Hence, there are finitely many $j$ such that $A_i^j\cup B_i^j\neq \partial K_i$. Let
\begin{equation*}
C_i=\bigcap_{j\neq i}{A_i^j\cup B_i^j}.
\end{equation*}
By considering $(n-1)$-dimensional Lebesgue measure, one can show that the closure of $C_i$ is $\partial K_i$. Therefore, to prove the lemma, it suffices to show that the statement is true for all $q\in C_i$.
\begin{figure}[!ht]
  \centering
  \subfigure[$K_i$ and $K_j$]{
    \includegraphics[scale=.8]{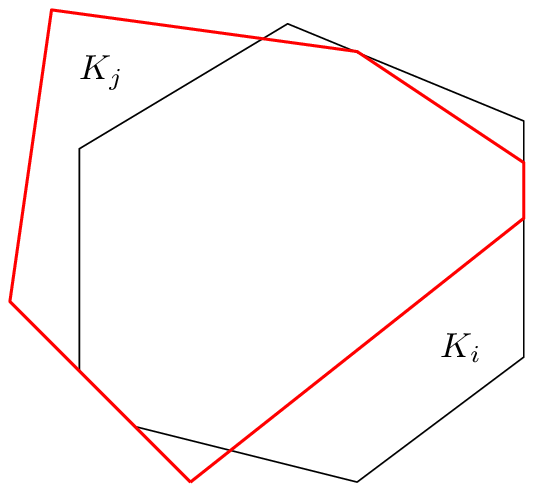}}
    \subfigure[points $p_1,p_2,p_3,p_4,p_5,p_6$]{
    \includegraphics[scale=.8]{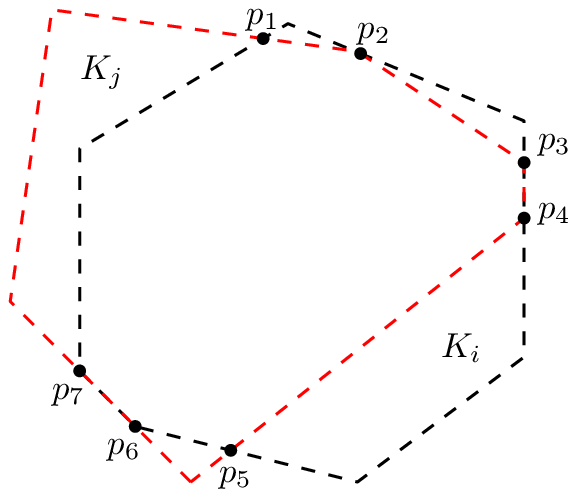}}
   \caption{$A_i^j=\partial K_i\setminus\{p_1,p_2,p_3,p_4,p_5\}$, $B_i^j=\{p_2,p_3,p_4\}$ and $\partial K_i\setminus(A_i^j\cup B_i^j)=\{p_1,p_5\}$}\label{ABij}
\end{figure}

Suppose that $q\in C_i\cap \inte(D)$. Denote by $\mathcal{F}(q)$ the collection of convex bodies $K_j$ containing the point $q$. We note that $\mathcal{F}(q)$ must be finite. Let
\begin{equation*}
\mathcal{F}'(q)=\{K_1,K_2,\ldots\}\setminus \mathcal{F}(q).
\end{equation*}
It is not hard to see that, there exists a positive real number $\varepsilon_0$ such that $B_{\varepsilon_0}(q)\cap K_j=\emptyset$, for any $K_j\in\mathcal{F}'(q)$. Since $q\in C_i\cap\inte(D)$, we may assume, without loss of generality, that $B_{\varepsilon_0}(q)\subset \inte(D)$ and for each $j$, we have $B_{\varepsilon_0}(q)\cap \partial K_j\cap\inte(K_i)=\emptyset$ or $B_{\varepsilon_0}(q)\cap \partial K_j\subset K_i$. Furthermore, we may also assume that for all $j$, both $B_{\varepsilon_0}(q)\cap K_j$ and $B_{\varepsilon_0}(q)\setminus \inte(K_j)$ are connected. For a unit vector $u$, we denote by $R(q,u)$ the ray parallel to $u$ and starting at $q$.
One can find a unit vector $u$ that satisfies
\begin{enumerate}
\item[(i)] $R(q,u)\cap K_i=\{q\}$ and $R(q,-u)\cap\inte(K_i)\neq\emptyset$,
\item[(ii)] there is a point $q'\in R(q,u)\cap B_{\varepsilon_0}(q)$ such that $q'\notin\partial K_j$ for all $j=1,2,\ldots$.
\end{enumerate}
Since $\{K_1,K_2,\ldots\}$ is a $k$-fold tiling of $D$ and $q'\in B_{\varepsilon_0}(q)\subset\inte(D)$, there exist exactly $k$ convex bodies $K_{i_1},\ldots,K_{i_k}$ such that $q'\in \inte(K_{i_j})$, $j=1,\ldots,k$.
We note that $\{K_{i_1},\ldots,K_{i_k}\}\subset\mathcal{F}(q)$ and $K_i\notin \{K_{i_1},\ldots,K_{i_k}\}$. Denote by $\tilde{\mathcal{F}}(q)$ the collection of convex bodies $K_j$ which contain the point $q$ as an interior point. If $\{K_{i_1},\ldots,K_{i_k}\}\subset\tilde{\mathcal{F}}(q)$, then it is easy to see that $K_{i_1}\cap\cdots\cap K_{i_k}\cap K_i$ must have an interior point, which is impossible, since $\{K_1,K_2,\ldots\}$ is a $k$-fold tiling. Now we suppose that $K_{i_{j_0}}\notin \tilde{\mathcal{F}}(q)$, for some $j_0\in\{1,\ldots,k\}$. It is clear that $q\in\partial K_{i_{j_0}}$. We will show that $\inte(K_i)\cap\inte(K_{i_{j_0}})=\emptyset$. Since $q\in C_i$, we have that $q\in A_i^{j_0}$ or $q\in B_i^{j_0}$. Recall that for each $j$, we have $B_{\varepsilon_0}(q)\cap \partial K_j\cap\inte(K_i)=\emptyset$ or $B_{\varepsilon_0}(q)\cap \partial K_j\subset K_i$. If $q\in B_i^{j_0}$, then $B_{\varepsilon_0}(q)\cap \partial K_{j_0}\subset K_i$. From this, one can deduce that $B_{\varepsilon_0}(q)\cap K_{i_{j_0}}\subset K_i$ which is impossible, since $q'$ is in $B_{\varepsilon_0}(q)\cap K_{i_{j_0}}$ but is not in $K_i$. Therefore $q\in A_i^{j_0}$. It is not hard to see that $\partial K_{i_{j_0}}$ divides $B_{\varepsilon_0}(q)$ into two parts, where one of them dose not contain any point of $\inte(K_i)$, we denote this part by $B'$. Since $q'\in \inte(K_{i_{j_0}})$, it is obvious that $q'$ and $B_{\varepsilon_0}(q)\cap K_{i_{j_0}}$ must be contained in the same part. Because $q'\in R(q,u)$, by the property (i) of the vector $u$, we see that $q'$ must lie in $B'$. Therefore $B_{\varepsilon_0}(q)\cap K_{i_{j_0}}$ is contained in $B'$, and hence  $B_{\varepsilon_0}(q)\cap K_{i_{j_0}}\cap \inte(K_i)=\emptyset$. By Lemma \ref{local to global}, we obtain $\inte(K_i)\cap\inte(K_{i_{j_0}})=\emptyset$. This completes the proof.
\end{proof}

\begin{cor}\label{touch cor}
Let $X$ be a discrete multisubset of $\mathbb{R}^n$ containing the origin, and let $k$ be a positive integer. Suppose that $P$ is a centrally symmetric convex polytope with centrally symmetric facets, and $G$ is its subfacet. Let $q$ be a point on a facet in $\mathcal{B}_P(G)$. If $P+X$ is a $k$-fold tiling, then  $\overline{\partial}X_P(G,F_P^+(G,q),q)$ and $\overline{\partial}X_P(G,F_P^-(G,q),q)$ are not empty.
\end{cor}

We will denote by $X_P^+(G,q)$ and $X_P^-(G,q)$ the sets $\overline{\partial}X_P(G,F_P^+(G,q),q)$ and $\overline{\partial}X_P(G,F_P^-(G,q),q)$, respectively. For example, in Fig. \ref{XP}, we have $X_P^+(G,q)=\{x_1\}$ and $X_P^-(G,q)=\{x_3\}$.

\section{Proof of Main Theorem}

\begin{lem}\label{main lem}
If a convex body $K$ is a twofold translative tile, then $K$ is a centrally symmetric polytope with centrally symmetric facets, such that each belt of $K$ contains four or six facets.
\end{lem}
\begin{proof}
By Theorem \ref{gravin thm}, we know that $K$ is a centrally symmetric polytope with centrally symmetric facets.

Let $G$ be an arbitrary subfacet of $K$. Recall that we denote by $\mathcal{B}_K(G)$ the belt of $K$ determined by $G$.
Let $\mathcal{B}_K(G)$ have $m$ pairs of opposite facets. We will show that $m\leq 3$. To do this, we shall suppose that $m\geq 4$, and obtain a contradiction. Denote by $K(G)$ the union of the facets which are not contained in $\mathcal{B}_K(G)$.

Suppose that $K+X$ is a twofold translative tiling, where $X$ is a multisubset of $\mathbb{R}^n$. Without loss of generality, we may assume that $0\in X$. It is not hard to see that, we can choose a point $q\in G$ to lie in none of $K(G)+x$, where $x\in X$.

First, we will show that $\overset{\circ}{X}_K(q)=\emptyset$. If there is a $x\in\overset{\circ}{X}_K(q)$, then by Corollary \ref{touch cor}, there exist $x_1\in X_K^+(G,q)$ and $x_2\in X_K^-(G,q)$. Obviously, we have $x_1\neq x_2$. Since $m\geq 4$, it is easy to see that both $\ang_K(G,q)+\ang_{K+x_1}(G,q)$ and $\ang_K(G,q)+\ang_{K+x_2}(G,q)$ are greater than $(m-1)\pi-(m-2)\pi=\pi$. Therefore $\Ang_{K+x_1}(G,q)$ and $\Ang_{K+x_2}(G,q)$ are not opposite angles, and hence $\ang_K(G,q)+\ang_{K+x_1}(G,q)+\ang_{K+x_2}(G,q)$ is greater than $(m-1)\pi+(m-3)\pi=2\pi$ (Fig. \ref{sum of non opposite angle}). This can be deduced that $(K+x)\cap(K+x_1)\cap(K+x_2)$ has an interior point which is impossible, since $K+X$ is a twofold tiling.

\begin{figure}[!ht]
  \centering
    \includegraphics[scale=.85]{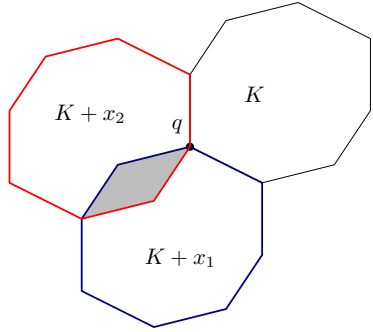}
   \caption{$\ang_K(G,q)+\ang_{K+x_1}(G,q)+\ang_{K+x_2}(G,q)>2\pi$}\label{sum of non opposite angle}
\end{figure}

We assert that $\dot{\partial} X_K(G,q)=\emptyset$. Suppose that $x\in \dot{\partial} X_K(G,q)$. Since $K$ is centrally symmetric, it is not hard to see that $\Ang_K(G,q)$ and $\Ang_{K+x}(G,q)$ are opposite angles (Fig. \ref{opposite angle}). By Corollary \ref{touch cor}, we can choose $x_1\in X_K^+(G,q)$ and $x_2\in X_K^-(G,q)$. Similar to the above argument, one obtains that $(K+x)\cap(K+x_1)\cap(K+x_2)$ has an interior point which is a contradiction.

\begin{figure}[!ht]
  \centering
    \includegraphics[scale=.85]{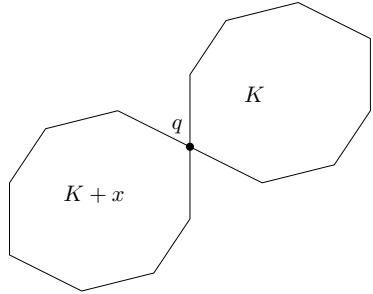}
   \caption{$\Ang_K(G,q)$ and $\Ang_{K+x}(G,q)$}\label{opposite angle}
\end{figure}

Now we will show that $X$ must be a usual set (not a multiset). If not, then we may assume that $0$ has multiplicity $2$. By Corollary \ref{touch cor}, one can choose $x_1\in X_K^+(G,q)$ and $x_2\in X_{K+x_1}^+(G,q)$ (see Fig. \ref{not multiset}). From the above discussion, we know that $\Ang_K(G,q),~\Ang_{K+x_1}(G,q)$ and $\Ang_{K+x_2}(G,q)$ cannot be opposite angles, hence $\ang_K(G,q)+\ang_{K+x_1}(G,q)+\ang_{K+x_2}(G,q)$ is greater than $2\pi$. This implies that $K\cap(K+x_2)$ has an interior point. We note that $x_2\neq 0$, and hence we obtain a contradiction.

\begin{figure}[!ht]
  \centering
    \includegraphics[scale=.9]{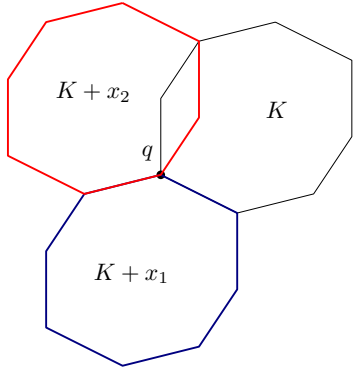}
   \caption{$x_1\in X_K^+(G,q)$ and $x_2\in X_{K+x_1}^+(G,q)$}\label{not multiset}
\end{figure}

We shall divide the remaining proof into the following two cases:
\begin{enumerate}
\item[(i)] \emph{Case $m\geq 6$:}~Since $\overset{\circ}{X}_K(q)=\emptyset$, we obtain $X_K(q)=\partial X_K(q)$. We note that $\ang_K(G,q)<\pi$ and for each $x\in\partial X_K(q)\setminus\{0\}$, we have $\ang_{K+x}(G,q)\leq\pi$. Because $K+X$ is a twofold tiling, so the cardinality of $\partial X_K(q)$ must be greater than $4$. On the other hand, since $m\geq 6$, we know that the sum of five (distinct) non-opposite angles is greater than $(m-1)\pi-(m-5)\pi=4\pi$. Therefore, the cardinality of $\partial X_K(q)$ cannot be greater than $4$, this is a contradiction.
\item[(ii)] \emph{Case $m=4$ or $5$:}~Similar to the above, we have that the cardinality of $\partial X_K(q)$ is greater than $4$. If there are two points $x,x'\in \partial X_K(q)$ such that $q$ lies in the relative interior of a facet of $K+x$ and also lies in the relative interior of a facet of $K+x'$, then $\ang_{K+x}(G,q)=\ang_{K+x'}(G,q)=\pi$, and hence $\sum_{z\in\partial X_K(q)}{\textstyle{\ang_{K+z}}(G,q)}$ is greater than
    $\textstyle{\ang_{K+x}(G,q)}+\textstyle{\ang_{K+x'}(G,q)}+(m-1)\pi-(m-3)\pi=4\pi$, which is impossible. Therefore, there is at most one point $x\in\partial X_K(q)$ such that $p$ lies in the relative interior of a facet of $K+x$. We choose $x_1\in X_K^+(G,q),~x_2\in X_{K+x_1}^+(G,q)$ and $x_3\in X_{K+x_2}^+(G,q)$. We note that $\ang_{K+x_1}+\ang_{K+x_2}<2\pi$ and $\ang_{K+x_1}(G,q)+\ang_{K+x_2}(G,q)+\ang_{K+x_3}(G,q)>2\pi$. Hence, one can prove that $E_K^+(G,q)\cap\inte(K+x_3)\neq \emptyset$ (Fig. \ref{m equal to 5}). Now we choose $q'\in E_K^+(G,q)\cap\inte(K+x_3)$ to lie in none of $K(G)+x$, where $x\in X$. Obviously, $x_3\in\overset{\circ}{X}_K(q')$. On the other hand, by using the same argument as the proof of $\overset{\circ}{X}_K(q)=\emptyset$, one obtains $\overset{\circ}{X}_K(q')=\emptyset$. This is a contradiction.
    \begin{figure}[!ht]
  \centering
    \includegraphics[scale=1]{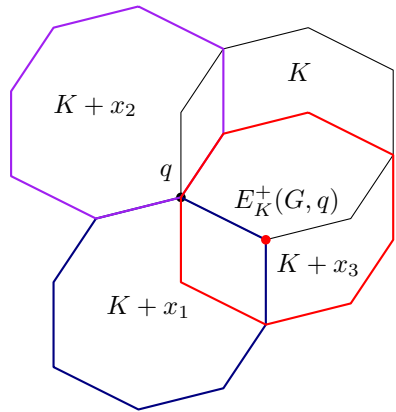}
   \caption{$E_K^+(G,q)\cap\inte(K+x_3)\neq \emptyset$}\label{m equal to 5}
\end{figure}
\end{enumerate}
Above all, we obtain $m\leq 3$.
\end{proof}

By Lemma \ref{main lem} and Theorem \ref{mcmullen thm}, one obtain Theorem \ref{main thm}.

\end{document}